\documentclass[10pt]{amsart}

\usepackage{amssymb,amsmath,amsthm}
\usepackage{tikz}
\usepackage{tikz-cd}
\usepackage[arc,all]{xy}
\usepackage{eucal}
\usepackage{color}

\usepackage{comment}

\usepackage[colorlinks=true,linkcolor=blue,urlcolor=black,bookmarks=false,bookmarksopen=false]{hyperref}

\theoremstyle{plain}
\newtheorem{cor}[subsection]{Corollary}
\newtheorem{lem}[subsection]{Lemma}
\newtheorem{thm}[subsection]{Theorem}
\newtheorem{prop}[subsection]{Proposition}
\newtheorem*{notation*}{Notation}

\theoremstyle{definition}

\theoremstyle{remark}
\newtheorem{rem}[subsection]{Remark}

\newcommand{\Z}{\mathbb{Z}}
\newcommand{\kk}{\mathbf{k}}

\newcommand{\lp}{_{(p)}}

\newcommand{\opn}{\operatorname}

\numberwithin{equation}{section}

\title{Cohomology of $BPU_n$ and rings of invariants of Weyl groups}

\author[F. Fan]{Feifei Fan}
\address{School of Mathematical Sciences, South China Normal University,
Guangzhou, 510631, P. R. China}
\email{fanfeifei@mail.nankai.edu.cn}

\author[J. Zha]{Jiaxi Zha}
\address{Department of Mathematics, Nankai University, No.94 Weijin Road, Tianjin 300071, P. R. China}
\email{1093913699@qq.com}

\author[Z. Zhang]{Zhilei Zhang}
\address{School of Science, Xi’an Shiyou University,
 Xi’an 710065, P. R. China}
\email{15829207515@163.com}

\author[L. Zhong]{Linan Zhong$^*$}
\address{Department of Mathematics, Yanbian University, Yanji 133000, P. R. China}
\email{lnzhong@ybu.edu.cn}

\thanks{$*$ Corresponding author}



\thanks{}

\begin{document}

\begin{abstract}
    Let $PU_n$ denote the projective unitary group of rank $n$ and $BPU_n$ be its classifying space, for $n>1$.
    Using the Serre spectral sequence associated to the fibration $BU_n\to BPU_n\to K(\mathbb{Z},3)$, we compute the integral cohomology group of $BPU_n$ in dimensions $\leq 14$. In addition, we determine the ring
    structure of $H^*(BPU_n;\mathbb{Z})$ up to dimension $13$ by computing the ring of
    invariants $H^*(BT_{PU_n})^W$ of the Weyl group action in dimensions $\leq 12$.
\end{abstract}

\maketitle

\section{Introduction}
Let $G$ be a compact connected Lie group. One of the main invariants associated
with $G$ is the cohomology of the classifying space $BG$. Let $T_G$ be a
maximal torus of $G$. We denote by $W$ the Weyl group $N_G(T_G)/T_G$ of $G$. Since $G$ is connected, the Weyl group action induced by conjugation in $G$ is homotopically trivial and so the action of $W$ on the cohomology of $BG$ is trivial. Therefore, we have the induced homomorphism of cohomology rings with coefficient ring $R$
\[H^*(BG;R)\to H^*(BT_G;R)^W,\]
where $H^*(BT_G;R)^W$ is the ring of invariants under the conjugation $W$-action. 

For rational coefficients, we have $H^*(BG;\mathbb{Q})\cong H^*(BT_G;\mathbb{Q})^{W}$ by the work of Borel \cite{Borel1953,Borel1955}. 
Similar identifications hold if the order of the Weyl group is invertible in the coefficient ring $R$. For general $G$ and $R$, the $R$ cohomology ring of $BG$ can be quite complicated.

In this paper we study the integral cohomology of $BG$ and the ring of invariants $H^*(BT_G;\mathbb{Z})^{W}$ for $G=PU_n$, the \emph{projective unitary group} obtained as the quotient group of the unitary group $U_n$ by its center $S^1$. 
The cohomology of $BPU_n$ plays a significant role in the study of the topological period-index problem (\cite{antieau2014period,antieau2014topological,gu2019topological,
gu2020topological}).
It is also crucial in the study of anomalies in particle 
physics (\cite{cordova2020anomalies,garcia2019dai}). Recently, Chen and Gu \cite{CG24} have found another interesting application of the cohomology of $BPU_n$ to the topological complexity problem in enumerative algebraic geometry.

Although the integral cohomology of $PU_n$ is fully determined in \cite{duan}, the cohomology of $BPU_n$ is not known for general $n$, and its calculation is known as a difficult problem in algebraic topology. For special values of $n$, the cohomology of $BPU_n$ has been studied in various works, such as Kono-Mimura \cite{kono1975cohomology}, Kono-Mimura-Shimada \cite{kono1975cohomology1},  Kono-Yagita \cite{Kono_Yagita}, Kameko-Yagita \cite{Kameko2008brown}, Toda \cite{toda1987cohomology}, Vezzosi \cite{vezzosi2000chow}, Vavpeti{\v{c}}-Viruel \cite{vavpetivc2005mod}, Vistoli \cite{vistoli2007cohomology}, and \cite{fan2024cohomology,Fan25BPUp} by the first named author. Summarizing these results, the ring structure of $H^*(BPU_n;R)$ for an arbitrary coefficient ring $R$ is determined for $n\leq 4$.

For an arbitrary integer $n\geq 5$, the cohomology of $BPU_n$ is only known in a finite range of dimensions. For $k\leq5$, $H^k(BPU_n;\Z)$ was calculated by Antieau and Williams \cite{antieau2014topological}. Recently, Gu made a breakthrough in this direction \cite{gu2019cohomology}, where the ring structure of $H^*(BPU_n;\Z)$ in dimensions $\leq 10$ for an arbitrary value of $n$ is determined. Subsequently, the first named author \cite{fan2024operators} improved Gu's result by determining the ring structure of $H^*(BPU_n;\Z)$ in dimensions $\leq 11$.


\begin{notation*}
To simplify notation we write $H^* (X)$ for the integral cohomology $H^*(X;\Z)$.  Given an abelian group $A$ and a prime number $p$, we let $A\lp$ denote the localization of $A$ at $p$, and let $_pA$ denote the $p$-primary subgroup of $A$.  In other words,
$_pA$ is the subgroup of $A$ consisting of torsion elements whose order is a power of $p$. One useful observation is that there exists a canonical isomorphism $_pH^*(-)\cong {_p[H^*(-)\lp]}$.  
Lastly, we use the notation 
$\gcd(m,n)$ to denote the greatest common divisor of two integers $m,n$.
\end{notation*}

In this paper, we determine the group structure of $H^*(BPU_n)$ in dimensions $\leq 14$ and the ring structure of $H^*(BPU_n)$ in dimensions $\leq 13$. We also compute $H^*(BT_{PU_n})^W$ in dimensions $\leq 12$. These calculations have some potential applications. For example, the determination of $H^*(BPU_n)$ in a larger range of dimensions can be used to extend the period-index results in \cite{antieau2014topological,gu2019topological,
gu2020topological} to higher dimensional CW-complexes. Furthermore, the calculation of $H^*(BT_{PU_n})^W$ turns out to be of some importance in the study of the topological complexity of enumerative problems in algebraic geometry \cite{CG24}.

Now we outline our strategy for studying $H^*(BPU_n)$ for arbitrary $n$. 
For a prime $p$, if $p\nmid n$, then the space $B(\Z/n)$ is 
$p$-locally contractible. Thus the Serre spectral sequence associated to the fiber sequence
\[B(\Z/n)\to BSU_n\to BPU_n \]
yields 
\begin{equation}
    \label{equ:p subgroup p not divide n}
    H^*(BPU_n)_{(p)}\cong H^*(BSU_n)_{(p)}=\Z_{(p)}[c_2, c_3, \dots, c_n ], 
\ p\nmid n.
\end{equation}
In other words, $_pH^*(BPU_n)=0$. Similarly, since $H^*(B(\Z/n);\mathbb{Q})=0$, we also have
\begin{equation}
\label{equ:torsion-free}
    H^*(BPU_n;\mathbb{Q})\cong
    H^*(BSU_n;\mathbb{Q})=
    \mathbb{Q} [c_2, c_3, \dots, c_n ].
\end{equation}
Hence, to determine the group structure of $H^*(BPU_n)$, it suffices to determine the $p$-primary subgroup $_pH^*(BPU_n)$ for $p\mid n$.

For any odd prime $p$, $_p H^* (BPU_n)$ in 
dimensions less than $2p+15$ have been completely determined
in a series of works by some authors of this paper and their collaborators \cite{gu_zzz,zzz,zhang2024,fan2024operators}. So, in order to determine $H^*(BPU_n)$ in
dimensions less than 15, it suffices to compute $_2H^*(BPU_n)$ in these dimensions, which is our first theorem.

\begin{thm}\label{thm: tor subgroup}
     For an integral $n\geq 2$, the torsion subgroup of $H^s(BPU_n;\Z)$ for $11<s<15$ is as follows: 
        \begin{gather*}
                H^{12}(BPU_n)_{tor} \cong \Z/\gcd(2,n)
                \oplus\Z/\gcd(5,n),\\
                H^{13}(BPU_n)_{tor}  \cong H^{14}(BPU_n)_{tor}\cong \Z/\gcd(2,n).
        \end{gather*}
\end{thm}

Let $K_n$ be the quotient ring of $H^*(BPU_n)$ by the ideal consisting of torsion elements and $K_n^i$ be the graded piece of $K_n$ in degree $i$. The group structure of $K_n$ is easily obtained from (\ref{equ:torsion-free}), but its ring structure is quite complicated. The representation of $K_n$ by generators and relations is unknown when $n\ge 5$ (see \cite{vezzosi2000chow} for $n=3$, and \cite{fan2024cohomology} for $n=4$). In fact, $K_n$ can be identified with the subring
\[\opn{Im}(Bq)^*\subset H^*(BU_n)\cong \mathbb{Z}[c_1,c_2,\dots,c_n],\]
where $Bq: BU_n\to BPU_n$ is the map induced by the quotient map $q:U_n\to PU_n$, as well as with the ring of invariants $H^*(BT_{PU_n})^W$ (see Theorem \ref{thm:d^0}). Hence, an element of $K_n$ is a polynomial in $\mathbb{Z}[c_1,c_2,\dots,c_n]$. 

By \eqref{equ:torsion-free}, there exist $e_i\in K_n^{2i}$, $2\leq i\leq n$, and a graded ring monomorphism $\Z[e_2,\dots,e_n]\to K_n$ such that the graded quotient group $K_n/\Z[e_2,\dots,e_n]$ is a torsion group (see Section \ref{sec: proof of thm 2} for details of the construction of $e_i$).
It can be shown that this monomorphism is an isomorphism up to dimension $10$. However, in dimension $12$, this map is no longer an isomorphism.

\begin{thm}\label{thm:cyc}
    $K_n^{12}/(e_2^3,e_3^2,e_4e_2,e_6)$ is a cyclic group of order $\lambda_n^3$, where $e_i=0$ if $i>n$, and
        \begin{equation*}
        \begin{split}
            \lambda_n=
            \begin{cases}
            n, & \text{if} \  n\not\equiv 2\mod 4,\\            
            \frac{n}{2}, & \text{if} \  n\equiv 2\mod 4.
            \end{cases}\
        \end{split}
    \end{equation*}
\end{thm}

Armed with the above theorems, we can further determine the ring structure of
$H^*(BPU_n)$ up to dimensions $13$. Recall that the cohomology rings of $BPU_n$
for $n=2,3,4$ are already known (cf. \cite{vistoli2007cohomology,fan2024cohomology}).
\begin{thm}\label{thm: ring structure}
	For an integer $n\geq5$, $H^*(BPU_n)$ in dimensions $\le 13$ is isomorphic to the following graded ring:
	\[\Z[e_2, \cdots, e_{j_n} , \alpha_6, x_1, y_{3,0}, y_{2,1},y_{5,0}]/I_n.\]
	Here, the degree of $e_i$ is $2i$, $j_n = min\{6, n\}$. 
 The degrees of $\alpha_6$, $y_{5,0}$ are $12$, and
 the degrees of $x_1$, $y_{3,0}$, $y_{2,1}$ are $3$, $8$, $10$ respectively.
 
The ideal $I_n$ is generated by
$$nx_1 ,\  \gcd(2,n)x^2_1 ,\  \gcd(p,n)y_{p,0}, \ \gcd(2,n)y_{2,1},\ e_2y_{3,0},\ e_5x_1,\ \delta(n)e_2x_1,$$
$$(\delta(n)-1)(y_{2,1}-e_2x^2_1),\ e_3x_1,\ \mu(n)e_4x_1,\ \lambda^3_n\alpha_6-be_4e_2-ce_3^2-de_2^3
,$$
where 
    \begin{equation*}
        \begin{split}
            \delta(n)=
            \begin{cases}
            2, & \text{if} \  n\equiv 2\mod 4,\\            
            1, & \text{otherwies,}
            \end{cases}\quad
            \mu(n)=
            \begin{cases}
            4, & \text{if}\  n\equiv 4\mod8,\\
            2, & \text{if}\  n\equiv 0\mod8,\\   
            1, & \text{otherwise,}
            \end{cases}
        \end{split}
    \end{equation*}
and $b,c,d\in\Z$ (see Remark \ref{rem:b,c,d}).
\end{thm}

\subsection*{Organization of the paper}
In Section \ref{sec:spectral_sequence}, we introduce the Serre spectral sequence ${^U\!E}$ which is the main tool used to compute the cohomology of $BPU_n$.  In Section \ref{sec: proof of thm 1}, we complete
the proof of Theorem \ref{thm: tor subgroup}
via explicit computations of relevant 
differentials in the spectral sequence. The proof of Theorem \ref{thm:cyc} is purely algebraic and is contained in Section \ref{sec: proof of thm 2}. Theorem \ref{thm: ring structure} is an easy consequence of Theorem \ref{thm: tor subgroup} and \ref{thm:cyc}.


\section{The spectral sequences}\label{sec:spectral_sequence}

In this section, we recall the construction and computational results of the Serre spectral sequence ${^U\!E}$. This spectral sequence played a crucial role in our computation. Additionlally, two auxiliary spectral sequences $^T\!E$ and $^K\!E$ are also introduced, from which we can determine some differentials in $^U\!E$.
\subsection{The Serre spectral sequence $^U\! E$}

Applying the classifying space functor $B$ to the short exact sequence of Lie groups
$$1 \to S^1 \to U_n \xrightarrow{q} PU_n \to 1,$$
we obtain the following fiber sequence:
$$BS^1 \to BU_n \xrightarrow{Bq} BPU_n.$$
Note that $BS^1$ has the homotopy type of the Eilenberg-Mac Lane space $K(\Z, 2)$, so there is an associated fiber sequence
\[
U: ~ BU_n \to BPU_n \xrightarrow{\chi} K(\Z, 3).
\]
We denote the Serre spectral sequence associated to this fibration by $^U\!E$ and we will use it to compute the cohomology of $BPU_n$. The $E_2$ page of
$^U\!E$ has the form
\[
^U\! E^{s, t}_{2} = H^{s}(K(\mathbb{Z},3);H^{t}(BU_{n})) \Longrightarrow H^{s+t}(BPU_{n}).
\] 
We summarize the cohomology of $K(\Z, 3)$ in low dimensions as follows.  The original reference is \cite{cartan19551955}, see also \cite[Proposition 2.14]{gu2019cohomology}.

\begin{prop}\label{prop: p local cohomology of KZ3 below 2p+5}
In degrees up to $15$, $H^*(K(\Z,3))$ is isomorphic to the following graded ring:
$$\Z[x_1,y_{2,1},y_{2,(0,1)}, y_{3,0}, y_{5,0}]/(2x_1^2,2y_{2,1},
2y_{2,(0,1)},3y_{3,0}, 5y_{5,0}),$$
where the degrees of $x_1, y_{3,0},y_{2,1},y_{5,0},y_{2,(0,1)}$ are $3,8,10,12,15$, respectively.
\end{prop}
\begin{rem}
    Here we use the same notations for the generators as in \cite{gu2019cohomology}. Sometimes we abuse notations and let these generators denote their images in $H^*(BPU_n)$ under the homomorphism induced by the map $\chi:BPU_n\to K(\Z,3)$ in the fiber sequence $U$.
\end{rem}

Also recall that
\[
    H^{*}(BU_{n}) = \mathbb{Z}[c_1,c_2,\dots,c_n],\ \mid c_i\mid =2i.
\]
Since $H^{*}(BU_{n})$ is torsion-free and is concentrated in even dimensions, we have
\[
    ^U\!E^{s, t}_{3}={^U\!E}^{s, t}_{2} \cong H^{s}(K(\mathbb{Z},3)) \otimes H^{t}(BU_{n}).
\]
The higher differentials of $^U\!E$ have the form
\begin{equation}\label{eq:degree change}
	d_r:{^U\!E}^{s,t}_r\to{^U\!E}^{s+r,t-r+1}_r.
\end{equation}

\subsection{The auxiliary spectral sequences $^T\! E$ and $^K\! E$}

Since the differentials in $^U\!E$ is difficult to compute directly, we use Gu's strategy that compares $^U\!E$ with two auxiliary spectral sequences, which have simpler differential behaviors and  
are introduced as follows.

Let $T^n$ be the maximal torus of $U_n$ with the inclusion denoted by
\[\psi: T^n\to U_n.\]
The normal subgroup of scalar matrices $S^1$ can also be considered as a subgroup of $T^n$. Passing to quotients over $S^1$, we have another inclusion of maximal torus
\[\psi': PT^n\to PU_n,\]
and an exact sequence of Lie groups
\[1\to S^1\xrightarrow{\varphi} T^n\to PT^n\to 1,\]
Here the inclusion map $S^1\to T^n$ can be identified as the diagonal map which we denote by $\varphi$.

Applying the classifying space functor and we obtain the fiber sequence 
\[
T: ~ BT^n \to BPT^n \to K(\mathbb{Z}, 3).
\]
$T$ is our first auxiliary fiber sequence.

Our second auxiliary fiber sequence $K$ is just the following path fibration:
\[
K: ~ K(\mathbb{Z}, 2) \simeq BS^1 \to * \to K(\mathbb{Z}, 3),
\]
where $*$ denotes a contractible space.

These fiber sequences fit into the following homotopy commutative diagram:
\[
    \begin{tikzcd}
        K\arrow[d,"\Phi"]:& BS^1\arrow[r]\arrow[d,"B\varphi"]& *\arrow[r]\arrow[d]&
        K(\Z,3)\arrow[d,"="]\\
        T:\arrow[d,"\Psi"]& BT^n\arrow[r]\arrow[d,"B\psi"]& BPT^n\arrow[r]\arrow[d,"B\psi'"]& K(\Z,3)\arrow[d,"="]\\
        U:& BU_n\arrow[r]& BPU_n\arrow[r]& K(\Z,3)
    \end{tikzcd}
\]

The Serre spectral sequences associated to $U$, $T$, and $K$ will be donoted by $^U\! E$, $^T\! E$ and $^K\! E$, respectively.  We also denote their corresponding differentials by ${^U\!d}_*^{*,*}$, ${^T\!d}_*^{*,*}$ and ${^K\!d}_*^{*,*}$, respectively.

We first describe the comparison maps between $^U\!E$, $^T\!E$ and $^K\!E$.

The induced homomorphisms between cohomology rings are as follows:
\[B\varphi^*:H^*(BT^n) = \Z[v_1,\cdots,v_n] \to H^*(BS^1) = \Z[v],\ v_i\mapsto v.\]
\[
    \begin{split}
        B\psi^*: H^*(BU_n) = \Z[c_1,\cdots,c_n] &\to H^*(BT^n) = \Z[v_1,\cdots,v_n],\\
        c_i &\mapsto \sigma_i(v_1,\cdots,v_n),
    \end{split}
\]
where $\sigma_i(x_1,\cdots,x_n)$ is the $i$th elementary symmetric polynomial in $n$ variables.

We also recall some important propositions regarding the higher differentials in $^K\!E$ and $^T\!E$.

\begin{prop}[\cite{gu2019cohomology}, Corollary 2.16]\label{prop:diff0}
 The higher differentials of ${^K\!E}_{*}^{*,*}$ satisfy
 \begin{equation*}
 \begin{split}
   &d_{3}(v)=x_1,\\
   &d_{2p^{k+1}-1}(p^k x_{1}v^{lp^{e}-1})=v^{lp^{e}-1-
   (p^{k+1}-1)}y_{p,k},\quad e > k\geq 0,\ \gcd(l,p)=1,\\
   &d_{r}(x_1)=d_{r}(y_{p,k})=0,\quad \textrm{for all }r,k>0
 \end{split}
\end{equation*}
and the Leibniz rule. 
\end{prop}

The differentials of ${^K\!E}_{*}^{*,*}$ and 
of ${^K\!T}_{*}^{*,*}$ are related by the following proposition.

\begin{prop}[\cite{gu2019cohomology}, Proposition 3.2]\label{pro:diff1}
    The differential  $^{T}\!d_{r}^{*,*}$, is partially determined as follows:
\[
 ^{T}\!d_{r}^{*,2t}(v_{i}^{t}\xi)={(B\pi_i)^{*}}({^K\!d}_{r}^{*,2t}(v^{t}\xi)),
\]
where $\xi\in {^{T}\!E}_{r}^{*,0}$, a quotient group of $H^*(K(\mathbb{Z}, 3))$, and $\pi_i: T^{n}\rightarrow S^1$ is the projection of the $i$th diagonal entry. In plain words, $^{T}\!d_{r}^{*,2t}(v_{i}^{t}\xi)$ is simply $^{K}\!d_{r}^{*,2t}(v^{t}\xi)$ with $v$ replaced by $v_i$.
\end{prop}

By comparing with the differentials in $^K\!E$, one could obtain the following results on differentials in $^T\!E$.

\begin{prop}
\label{prop: dT vn x1}
In the $2$-localized spectral sequence $^T\!E$, we have
\begin{enumerate}
    \item 
    $2v_n x_1, 4v_n ^3 x_1 ,2v_n^5 x_1 \in \opn{Im}
{^{T}\!d}_{3}$.
    \item 
    $^{T}\!d^{3,*}_{7}(2v_{n}^3 x_1) = y_{2,1}$.
\end{enumerate} 
\end{prop}

\begin{proof}
(1) From the first formula in Proposition \ref{prop:diff0} together with  Proposition \ref{pro:diff1}, we have $^{T}\!d_{3}
(v_n ^2)=2v_n x_1,\ ^{T}\!d_{3} (v_n ^4)=4v_n ^3 x_1 ,\ 
^{T}\!d_{3}(\frac{1}{3} v_n ^6)=2v_n^5 x_1$.

(2) It is proved by applying Proposition \ref{pro:diff1} and the second formula in Proposition \ref{prop:diff0}, taking 
$k=l=1, e=2$.
\end{proof}

The following proposition and its corollary are useful for our computations.
\begin{prop}[\cite{gu2019cohomology}, Proposition 3.3]
\label{prop: d3 of T}
For ${^U\!E}_*^{*,*}$, we have
    \begin{enumerate}
        \item The differential $^{T}\!d_{3}^{0,t}$ is given by the ``formal divergence''
        \[\nabla=\sum_{i=1}^{n}(\partial/\partial v_i): H^{t}(BT^{n})\rightarrow H^{t-2}(BT^{n}),\]
        in such a way that $^{T}\!d_{3}^{0,*}=\nabla(-)\cdot x_{1}.$ 
        \item The spectral sequence degenerates at ${{^T}\!E}^{0,*}_{4}$. Indeed, we have $^{T}\!E_{\infty}^{0,*}=$ $^{T}\!E_{4}^{0,*}=\operatorname{Ker}{^T\!d}_{3}^{0,*}=\mathbb{Z}[v_{1}-v_{n},\cdots, v_{n-1}-v_{n}]$.
    \end{enumerate}
\end{prop}

\begin{cor}[\cite{gu2019cohomology}, Corollary 3.4]\label{cor:d3}
Let $c_0 =1$. For $k\ge 1$, we have
    \[^{U}\! d_{3}^{0,*}(c_{k})=\nabla(c_{k})x_1=(n-k+1)c_{k-1}x_1.\]
\end{cor}

Another useful result by Crowley and Gu gives a nice description of $K_n$.
\begin{thm}[\cite{crowley2021h}, Theorem 1.3]\label{thm:d^0}
    There are natural isomorphisms 
    \[K_n\cong\opn{Ker}({^U\!d}_3^{0,*})={^U\!E}_4^{0,*}={^U\!E}_\infty^{0,*}\cong \opn{Im}(Bq)^*\cong H^*(BT_{PU_n})^W.\]
\end{thm}

\section{Proof of theorem \ref{thm: tor subgroup}}
\label{sec: proof of thm 1}
From \cite{gu_zzz,zzz,zhang2024} we know that in dimensions $12,13,14$, the only possible nontrivial $_pH^*(BPU_n)$ for an odd prime $p$ is $_5H^{12}(BPU_n)=\Z/\gcd(5,n)\{y_{5,0}\}$. So we only consider the $2$-primary subgroups.

In this section we provide an explicit calculation of $2$-primary subgroups of the cohomology of $BPU_n$ 
in dimensions $12$ to $14$ via the Serre spectral sequence $^U\!E$. Since $_pH^* (BPU_n)\cong {_p[H^*(BPU_n)_{(p)}]}$ for any prime $p$, it suffices to look at the $2$-localized Serre spectral sequence, where the $E_2$-page becomes
\[
    ({^U\!E}^{s,t}_{2})_{(2)}=H^{s}(K(\Z,3))_{(2)} \otimes H^{t}(BU_{n}).
\]

\begin{notation*}
	Recall that if $2\nmid n$, then $H^*(BPU_n)_{(2)}=0$. So throughout the rest of this section, we assume that $n>2$ is always even. We
    also use ${^U\!} E, {^T\!}E$ and $^K\!E$ to denote the corresponding $2$-localized Serre spectral sequences. 
\end{notation*}
We divide the proof of Theorem \ref{thm: tor subgroup} into two parts. Firstly, we compute the $2$-primary subgroups of $H^*(BPU_n)$ in
dimensions $12,13$ by computing the spectral sequence $^U\!E$. Secondly, 
we determine the $2$-primary subgroup of $H^{14}(BPU_n)$, using more topological arguments.

\subsection{The $2$-primary subgroups of $H^{12}(BPU_n)$ and $H^{13}(BPU_n)$}\label{sec: 12-13}
The determination of these groups needs the following
result of $^U\!E_{\infty}$-page.

\begin{lem}
    \label{lem: E inf}
    In the spectral sequence $^U\!E$,  we have 
    $${^U\!E}^{12,0}_{\infty}={^U\!E}^{13,0}_{\infty}=\Z/2,$$
    and
    $${^U\!E}^{10,2}_{\infty}={^U\!E}_{\infty}^{3,10}={^U\!}E_{\infty}^{6,6}={^U\!}E_{\infty}^{9,4}={^U\!}E_{\infty}^{12,2}=0.$$
\end{lem}

\begin{proof}[Proof of Theorem \ref{thm: tor subgroup} for degrees $12$, $13$]
The nontrivial entries in ${^U\!E}_2^{*,*}$ of
total degree 12 are ${^U\!E}_2^{0,12}$, ${^U\!E}_2^{6,6}$,
${^U\!E}_2^{10,2}$ and ${^U\!E}_2^{12,0}$.
By Lemma \ref{lem: E inf},
${^U\!E}_{\infty}^{6,6}={^U\!E}_{\infty}^{10,2}=0$. Hence, we have a short exact sequence of $\Z_{(2)}$-modules
$$0\to {^U\!E}_{\infty}^{12,0} \to H^{12}(BPU_n)_{(2)}
\to {^U\!E}_{\infty}^{0,12} \to 0.$$
Since ${^U\!E}_{\infty}^{0,12} \subset {^U\!E}_2^{0,12}$
is a free $\Z_{(2)}$-module, the above short exact 
sequence is split and we have
$$H^{12}(BPU_n)_{(2)}\cong {^U\!E}_{\infty}^{12,0} \oplus
{^U\!E}_{\infty}^{0,12},$$
from which and Lemma \ref{lem: E inf} we deduce
$${_2 H}^{12}(BPU_n)\cong {^U\!E}_{\infty}^{12,0}\cong \Z/2.$$

Similarly, the nontrivial entries in ${^U\!E}_2^{*,*}$ of
total degree 13 are ${^U\!E}_2^{3,10}$, ${^U\!E}_2^{9,4}$
and ${^U\!E}_2^{13,0}$.
By Lemma \ref{lem: E inf},
${^U\!E}_{\infty}^{3,10}={^U\!E}_{\infty}^{9,4}=0$. Hence, 
by Lemma \ref{lem: E inf},
we have
$${_2 H}^{13}(BPU_n)= H^{13}(BPU_n)_{(2)}={^U\!E}_{\infty}^{13,0}=\Z/2.$$    
\end{proof}

The proof of the lemma involves basic calculations, which we will now provide in detail.
\begin{proof}[Proof of Lemma \ref{lem: E inf}]
Consider the following complex in the $E_3$ page:
    $${^U\!}E_3 ^{3,8}\xrightarrow{^U\!d_3 ^{3,8}}
    {^U\!}E_3 ^{6,6}\xrightarrow{^U\!d_3 ^{6,6}}
    {^U\!}E_3 ^{9,4}\xrightarrow{^U\!d_3 ^{9,4}}
    {^U\!}E_3 ^{12,2}.$$
Using Proposition \ref{prop: p local cohomology of KZ3 below 2p+5} and Corollary \ref{cor:d3}, an immediate calculation shows that
        $${^U\!}E_4^{12,2}={^U\!}E_4^{9,4}={^U\!}E_4^{6,6}=0.$$
It follows that ${^U\!}E_{\infty}^{12,2}={^U\!}E_{\infty}^{9,4}={^U\!}E_{\infty}^{6,6}=0$.

To determine ${^U\!E_{\infty}^{3,10}}$, we first consider the following complex
\[^U\!E_3 ^{0,12}\xrightarrow{^U\!d_3 ^{0,12}}
{^U\!}E_3 ^{3,10}\xrightarrow{^U\!d_3 ^{3,10}}
{^U\!}E_3 ^{6,8},\]
where
\begin{equation*}\label{eq:basis}
    \begin{split}
        ^U\!E_3 ^{0,12} = & \ \Z_{(2)} \{c_6, c_5 c_1, c_4 c_2, c_3^2, c_4 c^2 _1, c_3 c_2 c_1, c_2^3, c_3 c_1^3, c_2^2 c_1^2,
        c_2 c_1^4, c_1^6 \},\\
        ^U\!E_3 ^{3,10} = & \ \Z_{(2)} \{c_5x_1, c_4 c_1x_1, c_3 c_2x_1, c_3 c^2 _1x_1,
        c_2^2 c_1x_1, c_2 c_1^3x_1, c_1^5x_1 \},\\
        ^U\!E_3 ^{6,8} = & \ \Z/2 \{c_4x_1^2, c_3 c_1x_1^2, c_2^2x_1^2, c_2 c^2_1x_1^2, c_1^4x_1^2 \}.
    \end{split}
\end{equation*}
An immediate computation shows that
$$\opn{Ker} {^U\!d}_3 ^{3,10}=\Z_{(2)}\{c_5x_1, 2c_4 c_1x_1, c_4 c_1x_1+c_3 c_2x_1, c_3 c^2 _1x_1, c_2^2 c_1x_1, 2c_2 c_1^3x_1, c_1^5x_1\}.$$
Using the two bases above, we have the following matrix corresponding to ${^U\!d}_3^{0,12}$
	\begin{equation*}
M_2 =
\begin{pmatrix}
n-5 & 0 & 0 & 0 & 0 & 0 & 0\\
n & n-4 & 0 & 0 & 0 & 0 & 0\\
0 & n-1 & n-3 & 0 & 0 & 0 & 0\\
0 & 0 & 2(n-2) & 0 & 0 & 0 & 0\\
0 & 2n & 0 & n-3 & 0 & 0 & 0\\
0 & 0 & n & n-1 & n-2 & 0 & 0\\
0 & 0 & 0 & 0 & 3(n-1) & 0 & 0\\
0 & 0 & 0 & 3n & 0 & n-2 & 0\\
0 & 0 & 0 & 0 & 2n & 2(n-1) & 0\\
0 & 0 & 0 & 0 & 0 & 4n & n-1\\
0 & 0 & 0 & 0 & 0 & 0 & 6n\\
\end{pmatrix}
\end{equation*}
which is row equivalent to 
\begin{equation*}
\begin{pmatrix}
1 & 0 & 0 & 0 & 0 & 0 & 0\\
0 & 1 & -1 & 0 & 0 & 0 & 0\\
0 & 0 & 2 & 0 & 0 & 0 & 0\\
0 & 0 & 0 & 1 & 0 & 0 & 0\\
0 & 0 & 0 & 0 & 1 & 0 & 0\\
0 & 0 & 0 & 0 & 0 & 2 & 0\\
0 & 0 & 0 & 0 & 0 & 0 & 1\\
0 & 0 & 0 & 0 & 0 & 0 & 0\\
0 & 0 & 0 & 0 & 0 & 0 & 0\\
0 & 0 & 0 & 0 & 0 & 0 & 0\\
0 & 0 & 0 & 0 & 0 & 0 & 0\\
\end{pmatrix},\quad\quad
\begin{pmatrix}
1 & 0 & 0 & 0 & 0 & 0 & 0\\
0 & 1 & -1 & 0 & 0 & 0 & 0\\
0 & 0 & 4 & 0 & 0 & 0 & 0\\
0 & 0 & 0 & 1 & 0 & 0 & 0\\
0 & 0 & 0 & 0 & 1 & 0 & 0\\
0 & 0 & 0 & 0 & 0 & 2 & 0\\
0 & 0 & 0 & 0 & 0 & 0 & 1\\
0 & 0 & 0 & 0 & 0 & 0 & 0\\
0 & 0 & 0 & 0 & 0 & 0 & 0\\
0 & 0 & 0 & 0 & 0 & 0 & 0\\
0 & 0 & 0 & 0 & 0 & 0 & 0\\
\end{pmatrix},
\end{equation*}
for \textbf{(a)} $n\equiv 2\mod 4$ and 
\textbf{(b)} $n\equiv 0\mod 4$, respectively. Therefore,
\begin{equation*}
    {^U\!}E_7^{3,10}={^U\!}E_4^{3,10}\cong\opn{Ker} {^U\!}d_3 ^{3,10}/\opn{Im} {^U\!}d_3 ^{0,12}\cong
    \begin{cases}
        0,\ &\text{ if } n\equiv2\mod4,\\
        \Z/2\{2c_4c_1x_1\},\ &\text{ if } n\equiv0\mod4,
    \end{cases}    
\end{equation*}
where the first equality comes from \eqref{eq:degree change} for degree reasons.

Now we analyze the differential ${^U\!}d_7 ^{3,10}:
{^U}\!E_7^{3,10}\rightarrow {^U}\!E_7^{10,4}$ for $n\equiv 0$ mod $4$, where
${^U}\!E_7^{10,4}\subset {^U}\!E_4^{10,4}= \Z/2 \{c_2 y_{2,1}, c_1^2 y_{2,1}\}$ due to the degree reason. Instead of computing it directly, we use the map $\Psi^*: {^U\! E}\to {^T\! E}$ of spectral sequences to consider its image in ${^T\! E}$. That is
\begin{equation}\label{eq:example}
	\begin{split}
		\Psi^*{^U\!d}_7^{3,10}(2c_4c_1x_1) 
             =\ & {^T\!d}_{7}^{3,10}\Psi^{*}(2c_{4}c_1x_1)\\
             =\ & {^T\!d}_{7}^{3,10}[2
             (\sum_{n\geq i_1 > \cdots> i_4 \geq 1}
             v_{i_1}\cdots v_{i_4})(\sum_{k=1}^n v_k)x_1].
	\end{split}
\end{equation}

For $1\le i \le n$, let $v'_i =v_i-v_{n}$. It follows from (2) of Proposition \ref{prop: d3 of T} that the $v'_i$'s are permanent cycles. Then we can rewrite \eqref{eq:example} as

\begin{equation*}
    \begin{split}
        & \Psi^*{^U\!d}_7 ^{3,10}(2c_{4}c_1x_1)\\
             =\ & {^T\!d}_{7}^{3,10}[2(\sum_{n\geq i_1 > \cdots> i_4 \geq 1}(v'_{i_1}+v_n)\cdots (v'_{i_4}+v_n))(\sum_{k=1}^n v_k'+nv_n)x_1]\\
             =\ & {^T\!d}_{7}^{3,10}[2(\sum_{n\geq i_1 > \cdots> i_4 \geq 1}\sum_{j=0}^{4}\sigma_j(v'_{i_1},\cdots,v'_{i_4})v_n^{4-j})
             (\sum_{k=1}^n v_k'+nv_n)x_1]\\
              =\ & {^T\!d}_{7}^{3,10}
              [2(\sum_{n\geq i_1 > \cdots> i_4 \geq 1} \sigma_1 (v'_{i_1},\cdots,v'_{i_4}))(\sum_{k=1}^n v_k')v_n^{3}x_1]\\
              & + {^T\!d}_{7}^{3,10}
              [2(\sum_{n\geq i_1 > \cdots> i_4 \geq 1} \sigma_2 (v'_{i_1},\cdots,v'_{i_4}))nv_n^{3}x_1]\\
              =\ & (\sum_{n\geq i_1 > \cdots> i_4 \geq 1} \sigma_1 (v'_{i_1},\cdots,v'_{i_4})) (\sum_{k=1}^n v_k')
                 y_{2,1},\\
\end{split}
\end{equation*}
where the $3$rd and $4$th equations is due to Proposition \ref{prop: dT vn x1}
and the fact that $y_{2,1}$ is $2$-torsion. Then we have
\begin{equation}
\label{equ:d r 2c4c1x1'}
    \begin{split}
             \Psi^*{^U\!d}_7 ^{3,10}(2c_{4}c_1x_1)
             =&\tbinom{n-1}{3} (\sum_{k=1}^n v'_k)^2 y_{2,1}=\tbinom{n-1}{3} (\sum_{k=1}^n v_k -nv_n)^2 y_{2,1}\\
             =&\tbinom{n-1}{3} (\sum_{k=1}^n v_k)^2  y_{2,1}=\Psi^* (\tbinom{n-1}{3} c_1^2 y_{2,1}).
    \end{split}
\end{equation}

Obviously, the  map 
$\Psi^{*}: {^U\!E}_2^{10,4}\to {^T\!E}_2^{10,4}$ is injective. Since ${^U}\!E_7^{10,4}\subset {^U}\!E_2^{10,4}$ and ${^T}\!E_7^{10,4}\subset {^T}\!E_2^{10,4}$,
the induced map $\Psi^{*}: {^U\!E}_7^{10,4}\to {^T\!E}_7^{10,4}$ is also injective. Thus,
\eqref{equ:d r 2c4c1x1'} shows that
\begin{equation}\label{equ: d7 2 c4c1 x1}
    {^U\!d}_7 ^{3,10}(2c_{4}c_1x_1)=\tbinom{n-1}{3} c_1^2 y_{2,1}.
\end{equation}
Since $y_{2,1}$ is $2$-torsion, we have $c_1^2y_{2,1}=e_2y_{2,1}$, where $e_2$ is the element in \eqref{equ: e2e3}. Note that both $e_2$ and $y_{2,1}$ are permanent
elements \cite{gu2019cohomology}.
Thereofore, the differential
\eqref{equ: d7 2 c4c1 x1} is nontrivial when $n\equiv0$ mod $4$. Hence ${^U}\!E_{\infty}^{3,10}={^U}\!E_8^{3,10}=0$.

Similarly to \eqref{equ: d7 2 c4c1 x1}, we have the following differentials (see \cite[Lemma 8.2]{fan2024operators} for explicit computation):
\begin{equation}\label{equ: d7 2 c4 x1 and d7 c2 c2 x1}
\begin{split}
    {^U\!d}_7 ^{3,8}(2c_{4}x_1)= & \ \tbinom{n-1}{3} c_1 y_{2,1},\\
    {^U\!d}_7 ^{3,8}(c_2^2 x_1)= & \ \tbinom{n}{2}c_1y_{2,1}.
\end{split}
\end{equation}
Since ${^U\!E}^{10,2}_3 =\Z/2 \{ c_1 y_{2,1} \}$, \eqref{equ: d7 2 c4 x1 and d7 c2 c2 x1} gives ${^U\!E}^{10,2}_{\infty}=0$.

Finally, by \cite[Proposition 6.4]{fan2024cohomology}, 
$x_1^4$, $y_{2,1}x_1\in H^*(BPU_n)$ are nonzero, which gives
${^U\!E}^{12,0}_{\infty}={^U\!E}^{12,0}_2=\Z/2 \{x_1^4\}$ and ${^U\!E}^{13,0}_{\infty}={^U\!E}^{13,0}_2 =\Z/2\{y_{2,1}x_1\}$.
\end{proof}

\subsection{The $2$-primary subgroup of $H^{14}(BPU_n)$}
Note that the nontrivial entries in ${^U\!E}_2^{*,*}$ of
total degree 14 are ${^U\!E}^{0, 14}_2$, ${^U\!E}^{6, 8}_2$, ${^U\!E}^{10, 4}_2$ and ${^U\!E}^{12, 2}_2$. Moreover, by Lemma \ref{lem: E inf}, 
${^U\!E}^{12, 2}_{\infty}=0$.

Consider the following complex in the $^U\!E_3$ page
\[^U\!E_3 ^{3,10}\xrightarrow{^U\!d_3 ^{3,10}}
{^U\!}E_3 ^{6,8}\xrightarrow{^U\!d_3 ^{6,8}}
{^U\!}E_3 ^{9,6}.\]
An immediate computation by Proposition \ref{prop: p local cohomology of KZ3 below 2p+5} and Corollary \ref{cor:d3} shows that
$${^U\!E}^{6,8}_{9}={^U\!E}^{6,8}_{4}\cong\Z/2\{c_2^2x_1^2\}.$$
By computing the kernel of ${^U\!d_3 ^{10,4}}:{^U\!}E_3 ^{10,4}\to 
{^U\!}E_3 ^{13,2}$, we have ${^U\!}E_4 ^{10,4}=\Z/2\{c_1^2y_{2,1}\}$. 
Recall that $c_1^2y_{2,1}=e_2y_{2,1}$ is a permanent cycle, since both $y_{2,1}$ and $e_2$ are permanent elements.
Due to the degree reason and the differential \eqref{equ: d7 2 c4c1 x1}, we have
\begin{equation}\label{eq:E{10,4}}
    {^U\!E}^{10,4}_{\infty}={^U\!E}^{10,4}_{8}\cong
        \begin{cases}
        0,\ &\text{ if } n\equiv0\mod4,\\
        \Z/2\{c_1^2y_{2,1}\},\ &\text{ if } n\equiv2\mod4. 
    \end{cases} 
\end{equation}
Hence, in order to determine $_2H^{14}(BPU_n)$, we only need to compute the
differential ${^U\!d}^{6, 8}_{9}(c_2^2x_1^2)$.

Firstly, suppose that $n\equiv 0$ mod $4$. We claim that ${^U\!d}_{9}^{6,8}(c_2^2x_1^2)=0$. 

Indeed, by Proposition \ref{prop: p local cohomology of KZ3 below 2p+5} and an easy calculation by Corollary \ref{cor:d3}, we have 
$${^U\!E}_{9}^{15,0}={^U\!E}_2^{15,0}=\Z/2\{x_1^5\}\oplus\Z/2\{y_{2,(0,1)}\}.$$ 
So it suffices to prove that $x_1^5$ and $y_{2,(0,1)}$ are linearly independent elements in $H^{15}(BPU_n)$.

\begin{prop}
    If $n\equiv0$ mod $4$, then $x_1^5$ and $y_{2,(0,1)}$ are linearly independent in $H^{15}(BPU_n)$.
\end{prop}
\begin{proof}
Consider the diagonal map $\Delta:BU_4\to BU_n$ given by
$U_4 \hookrightarrow U_n$:
\begin{equation*}
A\to 
    \begin{pmatrix}
         A & \cdots & 0 & 0\\
         0 & A & \cdots & \vdots\\
         \vdots & \cdots & \ddots & 0\\
         0 & \cdots & \cdots & A\\
    \end{pmatrix},
\end{equation*}
and denote the induced map $BPU_4\to BPU_n$ by $\Delta^{\prime}.$

Note that this diagonal map induces a homomorphism between the Serre spectral sequences of $H^*(BPU_n)$ and $H^*(BPU_4)$, such that its restriction on the bottom row of the $E_2$ pages is the identity. Thus it suffices to handle the case when $n=4$, which follows from \cite[Theorem 2.3]{fan2024cohomology}.
\end{proof}

\begin{cor}
    If $n\equiv0$ mod $4$, then $_2H^{14}(BPU_n)=\Z/2\{e_4x_1^2\}$, where $e_4$ is given in Lemma \ref{lem:H8}.
\end{cor}

\begin{proof}
    It remains to prove that $e_4x_1^2=c_2^2x_1^2$. This is because $x_1^2$ is $2$-torsion, and
    ${^U\!d}_3^{3,10}(c_2c_1^3x_1)= c_1^4x_1^2$, 
    ${^U\!d}_3^{3,10}(c_4c_1x_1)= c_3c_1x_1^2$.
\end{proof}

Next, we assume that $n\equiv2$ mod $4$.
\begin{lem}\label{lem:d_9{6,8}}
If $n\equiv2$ mod $4$, then ${^U\!d}_{9}^{6,8}(c_2^2x_1^2)=x_1^5+y_{2,(0,1)}$.
\end{lem}
\begin{proof}
The map $\chi: BPU_2\to K(\Z,3)$ in the fiber sequence $U$ induces a homomorphism
$$\chi^*:H^*(K(\Z,3);\Z/2)\to H^*(BPU_2;\Z/2).$$
It is well known that $H^*(BPU_2;\Z/2)\cong \Z/2[w_2,w_3]$, where $w_2,\ w_3$ are the universal Stiefel-Whitney classes. Moreover, it is also known that (see, for example \cite[Proposition 4.3]{fan2024cohomology})
$$H^*(K(\Z,3);\Z/2)\cong \bigotimes_{k\ge 0} \Z/2[x_{2,k};2^{k+2}+1]\otimes \Z/2[x_1;3].$$
Here $\Z/2[x;k]$ denotes the polynomial algebra over $\Z/2$ with one generator $x$ of degree $k$. We also abuse the notation $x_1$ here to denote its mod 2 reduction.

By \cite[Proposition 6.2]{fan2024cohomology}, we have $\chi^*(x_1)=w_3$, $\chi^*(x_{2,0})=w_2w_3$, and 
\[
    \begin{split}           
        \chi^*(x_{2,1})=&\ \chi^*Sq^4(x_{2,0})
        =Sq^4\chi^*(x_{2,0})=Sq^4(w_2w_3)\\
        =&\ Sq^1(w_2)Sq^3(w_3)+Sq^2(w_2)Sq^2(w_3)\\
        =&\ w_3^3+w_2^3w_3.
        \end{split}
\]
By \cite[Proposition 2.14]{gu2019cohomology}, the mod $2$ reduction $\rho:H^*(K(\Z,3))\to $ 
$H^*(K(\Z,3);\Z/2)$ maps $y_{2,(0,1)}$ to $x_1^2x_{2,1}+x_{2,0}^3$ . Therefore, $$\chi^*\rho(y_{2,(0,1)})=w_3^2(w_3^3+w_2^3w_3)+w_2^3w_3^3=w_3^5=\chi^*(x_1^5).$$ 
In other words, $x_1^5=y_{2,(0,1)}\in H^{15}(BPU_2)$. 

 By Theorem \ref{thm:d^0}, the differential ${^U\!d}_{15}^{0,14}$ is trivial. So due to the degree reason, we must have ${^U\!d}_{9}^{6,8}(c_2^2y_{2,0})=x_1^5+y_{2,(0,1)}$ in the Serre spectral sequence of $BPU_2$. 

Now we consider the following commutative diagram induced by the diagonal map $\Delta:BU_2\to BU_n$, 
$$\xymatrix{
{^U\!E}_{9}^{6,8}(n)\ar@{->}[d]^{\Delta^*\otimes id} \ar@{->}[r]^{{^U\!d}_{9}^{6,8}}
& {^U\!E}_{9}^{15,0}(n)\ar@{=}[d]\ar@{=}[r] & {^U\!E}_2^{15,0}(n)\ar@{=}[d]\\
{^U\!E}_{9}^{6,8}(2)\ar@{->}[r]^{{^U\!d}_9^{6,8}}
& {^U\!E}_{9}^{15,0}(2)\ar@{=}[r] & {^U\!E}_2^{15,0}(2)\\
}$$
By \cite[(7.4)]{gu2019cohomology}, $\Delta^*(c_1)=\frac{n}{2}c_1$ and $\Delta^*(c_2)=\frac{n}{2}c_2+\frac{n(n-2)}{8}c_1^2$. Thus the left vertical map is an isomorphism, then the lemma follows. 
\end{proof}

Combining Lemma \ref{lem:d_9{6,8}} with \eqref{eq:E{10,4}}, we have 
\begin{cor}
	If $n\equiv2$ mod $4$, then $_2H^{14}(BPU_n)=\Z/2\{e_2y_{2,1}\}$.
\end{cor}

\section{Proof of Theorem \ref{thm:cyc} and \ref{thm: ring structure}}
\label{sec: proof of thm 2}

First we construct the elements $e_i\in K_n^{2i}$ as mentioned in the discussion preceding Theorem \ref{thm:cyc}. We implicitly assume that $e_i=0$ for $i>n$. 

The construction of $e_i$ is an inductive procedure, choosing $e_2$ as a generator of the infinite cyclic group $K_n^2\cong\mathbb{Z}$. Suppose that $e_i\in K_n^{2i}$, $i\leq m$, are chosen such that $\mathbb{Z}[e_2,\dots,e_m]\to K_n$ is an embedding and the graded quotient group $Q_{n,m}:=K_n/\mathbb{Z}[e_2,\dots,e_m]$ satisfies that $Q_{n,m}^{2i}$ is a torsion group for $i\leq m$. Then by \eqref{equ:torsion-free}, we have $Q_{n,m}^{2(m+1)}\cong\mathbb{Z}\oplus A$ for some torsion group $A$, and we choose $e_{m+1}$ to be a generator of the summand $\mathbb{Z}$. By the construction procedure, each $e_i$ is a polynomial in $c_1,\dots,c_i$ with relatively prime coefficients and with a nonzero multiple of $c_i$ as a term.

The explicit formulas of $e_2,e_3$ are given in \cite{gu2019cohomology}. 
\begin{equation}
\label{equ: e2e3}
    \begin{split}
        e_2 = & \ \gcd(2,n-1)^{-1}[2nc_2-(n-1)c_1^2],\ \text{ and}\\
        e_3 = & \ g_3^{-1}[3n^2c_3-3n(n-2)c_2c_1+(n-1)(n-2)c_1^3],
    \end{split}
\end{equation}
where $g_3=\gcd(3,n-1) \gcd(3,n-2)
\gcd(4,n-2)$.

\begin{lem}\label{lem:H8}
$K_n^8$ is spanned by the following two elements:
\begin{equation*}
    \begin{split}
        e_2^2 = & \ \gcd(2,n-1)^{-2}[4n^2c_2^2-4n(n-1)c_2c_1^2+(n-1)^2c_1^4],\\
        e_4= & \ \gcd(3,n)^{-1}[nc_4-(n-3)c_3c_1
        -\frac{1}{2}(n^2+n+1)(n-2)(n-3)c_2^2\\
        & +\frac{1}{2}n^2(n-2)(n-3)
        c_2c_1^2-\frac{1}{8}n(n-1)(n-2)(n-3)c_1^4].
    \end{split}
\end{equation*}
\end{lem}
\begin{proof}
Recall that $K_n^8\cong \opn{Ker}{^U\!d}_3^{0,8}$.
Thus, we only need to solve the equation
\[{^U\!d}_3^{0,8}(a_1c_4+a_2c_3c_1+a_3c_2^2+a_4c_2c_1^2+a_5c_1^4)=0,\ 
a_i\in\Z.\]
By Corollary \ref{cor:d3}, the above equation gives $(n-3)a_1 +na_2=0$, which implies that $\frac{n}{\gcd(3,n)}$ divides $a_1$. It is immediate to verify that ${^U\!d}_3^{0,8}(e_4)=0$ and $a_1=\frac{n}{\gcd(3,n)}$ for $e_4$. Therefore we have $\opn{Ker}{^U\!d}_3^{0,8}=\Z\{e_2^2, e_4\}$. Note that $a_1=0$ if and only if $a_2=0$, and
the coefficients of $e_2^2$ are coprime so that
$\opn{Ker}{^U\!d}_3^{0,8}\cap \Z[c_3,c_2,c_1]=\Z\{e_2^2\}$.
\end{proof}

Since the coefficients of $e_2e_3$ are coprime, it is a generator of the group $K_n^{10}\cong\mathbb{Z}^2$. Therefore, we have
\begin{lem}\label{lem:H10}
The elements $e_2e_3$, $e_5$ span $K_n^{10}$.
\end{lem}


Next we show that the quotient group in Theorem \ref{thm:cyc} is a cyclic group. 
Let $K_n'=K_n^{12}\cap \Z[c_1,c_2,c_3,c_4]$ and
$L_n=K_n'/\Z\{e_2^3,e_3^2,e_2e_4\}$.

\begin{prop}\label{prop:cyclic}
		The $p$-primary subgroup $_pL_n$ of $L_n$ is a cyclic group for any prime $p$. Moreover, $_pL_n=0$ if $p\nmid n$ or $p=2$, $n\equiv2$ mod $4$. 
	\end{prop}
\begin{proof}
		If $p\nmid n$, by \eqref{equ:p subgroup p not divide n}, we have
            the ring isomorphism
            \begin{equation*}
    _pH^*(BPU_n)\cong
    {_pH}^*(BSU_n)=
    \Z_{(p)} [c_2, c_3, \dots, c_n ].
\end{equation*}
By \eqref{equ: e2e3} and Lemma \ref{lem:H8}, and restricting this isomorphism to degree $12$, we obtain
\[K_n^{\prime}\otimes\Z_{(p)}= \Z_{(p)}\{e_4e_2,e_3^2,e_2^3\},\]
and	then the statement for $p\nmid n$ follows. 

If $p$ is a prime divisor of $n$, straightforward computation shows that $e_2\equiv xc_1^2\mod p$ for some $x\not\equiv 0$ mod $p$ and
\[e_3\equiv\begin{cases}
yc_1^3 \text{ for some } y\not\equiv 0, &\text{ if } p\neq 2,\text{ or }p=2\text{ and }n=4m,\\
c_3+c_1^3,&\text{ if } p=2 \text{ and } n=4m+2,
\end{cases}\mod p.\]
Moreover, $e_4\equiv ac_4+bc_3c_1+\text{other terms}$ mod $p$, where $a$ and $b$ can not be both zero. 

Let $\rho:K_n^{\prime}\to K_n^{\prime}\otimes\Z/p$ be the mod $p$ reduction map. Then $\rho(e_2^3)$, $\rho(e_3^2)$, $\rho(e_2e_4)$ are linearly independent over $\Z/p$ if $p=2$ and $n=4k+2$. This implies that $_pL_n=0$ in this case.
		
For the remaining cases of $p$ and $n$, there is only one relation, up to scalar multiplication, between $\rho(e_2^3)$, $\rho(e_3^2)$, $\rho(e_4e_2)$. That is 
$x^3\rho(e_3^2)-y^2\rho(e_2^3)=0$. Suppose on the contrary that
\[_pL_n=\Z/p^{k_1}\oplus\cdots\oplus\Z/p^{k_s},\ \ k_1\geq\cdots\geq k_s\geq 1,\ s>1.\]
Let $b_i\in K_n^{\prime}$ be an element such that its image in $L_n$ generates  the $\Z_{p^{k_i}}$ summand. Then in $K_n^{\prime}$, $p^{k_i}b_i=u_ie_2^3+v_ie_3^2+w_ie_4e_2$ for some $u_i,v_i,w_i\in\Z$, where one of
$u_i,v_i, w_i$ is not divided by $p$, so that
$\rho(u_ie_2^3+v_ie_3^2+w_ie_4e_2)=0$. 

The above analysis shows that $p\mid w_i$ and 
$u_ie_2^3+v_ie_3^2\equiv r_i(x^3e_3^2-y^2e_2^3)$ mod $p$ for some  $r_i\in\Z$ with $p\nmid r_i$.
Let 
\[f=r_2(u_1e_2^3+v_1e_3^2+w_1e_4e_2)-r_1(u_2e_2^3+v_2e_3^2+w_2e_4e_2)\in \Z\{e_2^3,e_3^2,e_4e_2\}.\] 
Then $f\equiv 0 \mod p$, which means that $\frac{1}{p}f\in\Z\{e_2^3,e_3^2,e_4e_2\}$. This implies that 
$$p^{k_1-1}r_2b_1-p^{k_2-1}r_1b_2\in \Z\{e_2^3,e_3^2,e_4e_2\}\subset K_n^{\prime},$$ 
i.e. $p^{k_1-1}r_2b_1-p^{k_2-1}r_1b_2=0$ in $L_n$. This is a contradiciton because $p\nmid r_1, r_2$. 
\end{proof}

Let $\mathcal{P}=\{p:\text{ $p$ is a prime divisor of $n$}\}$ and let $\mathcal{S}$ be the set of integers relatively
prime to the elements of $\mathcal{P}$. The set $\mathbf{r}$ of rational numbers $m/s, s\in\mathcal{S}$, is a subring of $\mathbb{Q}$. Now we let $\kk=\mathbf{r}\otimes \Z[\frac{1}{2}]$ if $n\equiv2$ mod $4$, and $\kk=\mathbf{r}$ otherwise.  

By Proposition \ref{prop:cyclic}, $L_n$ is a cyclic group. If $n\equiv0$ mod $4$, its order is coprime with any prime $p$ such that $p\nmid n$. If $n\equiv 2$ mod $4$, its order is coprime with any prime $p$ such that $p\nmid n$ or $p=2$. Thus, we have an isomorphism
$L_n\cong L_n\otimes \kk$.

So, in order to determine the order in Theorem \ref{thm:cyc}, it suffices to compute the order under the coefficient $\kk$.

\begin{lem}
    \label{lem: Kn' cap c123}
    $(K_n^{12}\otimes\kk) \cap \kk[c_1,c_2,c_3]=\kk\{\beta_6,e_2^3\}$, where
    \begin{equation*}\label{eq: beta6}
        \begin{split}
            \beta_6=& \frac{\gcd(3,n)}{\gcd(2,n)^2}[n^2c_3^2
            -2n(n-2)c_3c_2c_1+\frac{8n(n-2)^2}{9(n-1)}c_2^3\\
            & +\frac{2(n-1)(n-2)}{3}c_3c_1^3-\frac{(n-2)^2}{3}c_2^2c_1^2].
        \end{split}
    \end{equation*}
\end{lem}
\begin{proof}
Note that $(K_n^{12}\otimes\kk) \cap \kk[c_1,c_2]\cong\kk$ and
the coefficients of $e_2^3$ are coprime. Thus,
$$(K_n^{12}\otimes\kk) \cap \kk[c_1,c_2]=\kk\{e_2^3\}.$$

Now we consider  $(K_n^{12}\otimes\kk) \cap \kk[c_1,c_2,c_3]$,
which is the solution space of 
$${^U\!d}_3^{0,12} (a_1c_3^2+a_2c_3c_2c_1+a_3c_2^3+a_4c_3c_1^3+a_5c_2^2c_1^2
+a_6c_2c_1^4+a_7c_1^6)=0.$$
It is equivalent to the following equations by Corollary \ref{cor:d3},
\begin{equation*}
    \begin{split}
        \begin{cases}
            2(n-2)a_1 +na_2 =0,\\
            (n-1)a_2 +3na_4 =0,\\
            \cdots\\
        \end{cases}
    \end{split}
\end{equation*}  
The first two equations imply $\gcd(2,n)^{-2}\gcd(3,n)n^2\mid a_1$. Moreover, $a_1=0$ if and only if $a_2=a_4=0$. Hence, if there exists a
solution $\beta_6$ such that $a_1=\gcd(2,n)^{-2}\gcd(3,n)n^2$, then we must
have
$$(K_n^{12}\otimes\kk) \cap \kk[c_1,c_2,c_3]=\kk\{\beta_6,e_2^3\}.$$
It is immediate to verify that $\beta_6$ in the lemma is such a solution. 
\end{proof}

\begin{lem}
    \label{lem: Kn'}
    $(K_n^{12}\otimes\kk) \cap \kk[c_1,c_2,c_3,c_4]=\kk\{\alpha_6,\beta_6,e_2^3\}$, where
    \begin{equation*}
        \label{eq: alpha6}
        \begin{split}
            \alpha_6=& \frac{1}{\gcd(2,n-1)}[2nc_4c_2-(n-1)c_4c_1^2-
            \frac{3n(n-3)}{2(n-2)}c_3^2\\
            & +(n-3)c_3c_2c_1-
            \frac{(n-2)(n-3)}{3(n-1)}c_2^3].
        \end{split}
    \end{equation*}
\end{lem}
\begin{proof}
$(K_n^{12}\otimes\kk) \cap \kk[c_1,c_2,c_3,c_4]$ is the solution space of
$${^U\!d}_3^{0,12} (a_0c_4c_2+a_0'c_4c_1^2+a_1c_3^2+a_2c_3c_2c_1+a_3c_2^3+a_4c_3c_1^3+a_5c_2^2c_1^2
+a_6c_2c_1^4+a_7c_1^6)=0.$$
By Corollary \ref{cor:d3}, the above equation gives $(n-1)a_0+2na_0'=0$, which implies $2n/\gcd(2,n-1)\mid a_0$. Clearly $a_0=0$ if and only if $a_0'=0$. Hence, if there exists a
solution $\alpha_6$ such that $a_0=2n/\gcd(2,n-1)$, then we must
have
$$(K_n^{12}\otimes\kk) \cap \kk[c_1,c_2,c_3,c_4]=\kk\{\alpha_6,\beta_6,e_2^3\}.$$
It is immediate to verify that $\alpha_6$ in the lemma is such a solution. 
\end{proof}

As a consequence of Lemma \ref{lem: Kn' cap c123} and \ref{lem: Kn'}, we have the following two equations. For some $b_1,b_2,b_3\in\kk$,
\begin{equation}
    \label{eq:e3 2}
        e_3^2=\gcd(3,n)\lambda_n^2\beta_6+b_1e_2^3,
\end{equation}
\begin{equation}
    \label{eq:e2e4}
        \frac{\lambda_n}{n}e_4e_2=\frac{\lambda_n}{\gcd(3,n)}
        \alpha_6+b_2\beta_6+b_3e_2^3,
\end{equation}
where $\lambda_n$ is defined in Theorem \ref{thm:cyc}.

\begin{proof}[Proof of Theorem \ref{thm:cyc}]
It suffices to show that the order of $L_n$ is $\lambda_n^3$.
Recall that $L_n\cong L_n\otimes\kk$. So we compute the order under the coefficient $\kk$. Let $\bar{\alpha}_6,\bar{\beta}_6\in L_n\otimes\kk$ be the image of $\alpha_6,\beta_6$ respectively.
By \eqref{eq:e3 2} and \eqref{eq:e2e4}, $\bar{\alpha}_6\in L_n\otimes\kk$ is $\lambda_n^3$-torsion, and the order of $\bar{\beta}_6\in L_n\otimes\kk$  is $\gcd(3,n)\lambda_n^2$ . Now we claim that the order of $\bar{\alpha}_6$ is $\lambda_n^3$. 

Note that for any prime $p$, if $n\equiv 2$ mod $4$, $p\mid \lambda_n$ is equivalent to $p\mid n$ and $p\neq 2$. If $n\not\equiv 2$ mod $4$, $p\mid  \lambda_n$ is equivalent to $p\mid n$.

We consider the following two cases:

(1) If $3\nmid n$ or $9\mid n$, taking both sides of (\ref{eq:e2e4}) modulo any prime $p\mid \lambda_n$, we obtain
\[\frac{3(1-n)\lambda_n}{\gcd(3,n)\gcd(2,n-1)n}c_3c_1^3+\text{other terms}\equiv b_2\beta_6+b_3c_1^6 \mod p.\]
Since $\frac{3(1-n)\lambda_n}{\gcd(3,n)\gcd(2,n-1)n}\not\equiv0$ mod $p$,
we must have $p\nmid b_2$. In other words, $b_2$ and $\lambda_n$ are coprime. Hence, by 
\eqref{eq:e3 2} and \eqref{eq:e2e4}, the order of $\bar{\alpha}_6$ is $\lambda_n^3$.

(2) If $3\mid n$ but $9\nmid n$, we can similarly conclude that $p\nmid b_2$ for any prime $p\mid \lambda_n$ with $p\neq 3$. 
For prime $p=3$, if $3\nmid b_2$, we have done. Otherwise, by ($\ref{eq:e2e4}$), we obtain
\[\frac{\lambda_n}{n}e_4e_2= \frac{\lambda_n}{3}(\alpha_6-\beta_6)+(b_2+\frac{\lambda_n}{3})\beta_6+b_3e_2^3,\]
then $3\nmid b_2+\frac{\lambda_n}{3}$. Replace $\alpha_6$ by $\alpha_6-\beta_6$, then the order of $\bar{\alpha}_6$ is still $\lambda_n^3$.

The above claim shows that the order of $\bar{\beta}_6$ divides the order of $\bar{\alpha}_6$. 
By Lemma \ref{lem: Kn'}, $(K_n^{12}\otimes\kk) \cap \kk[c_1,c_2,c_3,c_4]$ can be spanned by $\alpha_6,\beta_6,e_2^3$. So we have
$$\lambda_n^3(K_n^{12}\otimes\kk) \cap \kk[c_1,c_2,c_3,c_4]\subset \kk\{e_4e_2,e_3^2,e_2^3\}.$$ 
Therefore
$L_n$ is a cyclic group generated by $\bar{\alpha}_6$ whose order is $\lambda_n^3$.
\end{proof}
  
\begin{rem}\label{rem:b,c,d}
  We can explicitly determine the integers $b,c,d$ through a tedious and routine calculation.
  For example, when $n=5$, let 
  $$\alpha_6=5c_4c_2-2c_4c_1^2-15c_3^2+16c_3c_2c_1+26c_2^3-4c_3c_1^3
  -36c_2^2c_1^2+15c_2c_1^4-2c_1^6.$$
  Then we have $125\alpha_6=25e_4e_2-3e_3^2+119e_2^3.$

  \begin{cor}\label{cor:Kn 12}
  For any $n\geq 2$, $K_n$ in degrees $\leq 12$ is isomorphic to the following graded ring 
     $$\Z[e_2,e_3,e_4,e_5,e_6,\alpha_6]/(\lambda_n^3\alpha_6-be_4e_2-ce_3^2-de_2^3)\ \text{ for some } b,c,d\in\Z,$$ where $e_i=0$ for $i>n$.
\end{cor}
  
\end{rem}

\begin{proof}[Proof of Theorem \ref{thm: ring structure}]
Recall that $K_n\cong {^U\!E}_\infty^{0,*}$.
Choose elements $\tilde{e}_i\in H^{2i}(BPU_n)$, $2\leq i\leq \min\{6,n\}$, such that their images in $^U\!E_\infty^{0,2i}$ are $e_i$ and the product relations in \cite[Theorem 1.3]{fan2024operators} are satisfied. Also, choose $\tilde\alpha_6\in H^{12}(BPU_n)$ such that its image in $^U\!E_\infty^{0,12}$ is $\alpha_6$. 

Now we prove the additional product relations in Theorem \ref{thm: ring structure}. $\tilde e_2y_{3,0}=0$ is clear, since $\tilde e_2y_{3,0}=0$ is $3$-torsion, but $H^{12}(BPU_n)$ has no $3$-torsion elements. By Theorem \ref{thm: tor subgroup} and its proof, $\tilde e_5x_1\in\Z/\gcd(2,n)\{x_1y_{2,1}\}$. So replacing $\tilde e_5$ by $\tilde e_5+y_{2,1}$ if necessary, we get $\tilde e_5x_1=0$.

It remains to verify that in $H^*(BPU_n)$
$$\alpha:=\lambda_n^3\tilde\alpha_6-b\tilde e_4\tilde e_2-c\tilde e_3^2-d\tilde e_2^3=0.$$
By Corollary \ref{cor:Kn 12}, $\alpha$ is a torsion element. So it has the form 
$\alpha=ux_1^4+vy_{5,0}$ by Theorem \ref{thm: tor subgroup} and its proof, where $u\in\Z/\gcd(2,n)$, $v\in\Z/\gcd(5,n)$. Since $nx_1=\tilde e_2x_1=\tilde e_3x_1=0$, we have $\alpha x_1=ux_1^5+vy_{5,0}x_1=0$ if $n\not\equiv 2$ mod $4$ by the definition of $\lambda_n$. It follows that $u=0$ by \cite[Proposition 6.4]{fan2024cohomology}, and that $v=0$ by \cite[Lemma 7.3]{fan2024operators}, and so $\alpha=0$ in this case. 
For the case $n\equiv 2$ mod $4$, the same reasoning shows that
$\alpha x_1\in\Z/\gcd(2,n)\{x_1^5\}$ so that $\alpha\in \Z/\gcd(2,n)\{x_1^4\}$. If $\alpha=x_1^4$, we simply replace $\tilde \alpha_6$ by $\tilde \alpha_6+x_1^4$. Then we still have $\alpha=0$, noting that $\lambda_n$ is odd and $x_1^4$ is $2$-torsion.
\end{proof}

\subsection*{Acknowledgments} 
The first named author is supported by the National Natural Science Foundation of China (Grant No. 12271183) and by the GuangDong Basic and Applied Basic Research Foundation (Grant No. 2023A1515012217). The thirdly and fourth named author are supported by the National Natural Science Foundation of China (Grant No. 12261091). The fourth named author was also supported by the National Natural Science Foundation of China (Grant No. 12001474). 

\bibliographystyle{plain}
\bibliography{ref}

\begin{thebibliography}{10}

\bibitem{antieau2014period}
Benjamin Antieau and Ben Williams.
\newblock The period-index problem for twisted topological {K}-theory.
\newblock {\em Geometry \& Topology}, 18(2):1115--1148, 2014.

\bibitem{antieau2014topological}
Benjamin Antieau and Ben Williams.
\newblock The topological period--index problem over $6$-complexes.
\newblock {\em Journal of Topology}, 7(3):617--640, 2014.

\bibitem{Borel1953}
Armand Borel.
\newblock Sur la cohomologie des espaces fibres principaux et de espaces homogenes de groupes de {L}ie compacts.
\newblock {\em Annals of mathematics}, 57(1):115--207, 1953.

\bibitem{Borel1955}
Armand Borel.
\newblock Topology of {L}ie groups and characteristic classes.
\newblock {\em Bulletin of the American Mathematical society}, 61(5):397--432, 1955.

\bibitem{cartan19551955}
Henri Cartan and Jean-Pierre Serre.
\newblock S{\'e}minaire {H}enri {C}artan vol. 7 no.1.
\newblock 1954-1955.

\bibitem{CG24}
Weiyan Chen and Xing Gu.
\newblock Topological complexity of enumerative problems and classifying spaces of {$PU_n$}.
\newblock arXiv:2411.00497, 2024.

\bibitem{cordova2020anomalies}
Clay C\'ordova, Daniel~S. Freed, Ho~Tat Lam, and Nathan Seiberg.
\newblock Anomalies in the space of coupling constants and their dynamical applications {II}.
\newblock {\em SciPost Physics}, 8(1):Paper No. 002, 32pp, 2020.

\bibitem{crowley2021h}
Diarmuid Crowley and Xing Gu.
\newblock On {$H^*(BPU_n;\mathbb{Z})$} and {Weyl} group invariants.
\newblock {\em arXiv:2103.03523}, 2021.

\bibitem{duan}
Haibao Duan.
\newblock The cohomology of projective unitary groups.
\newblock {\em Trudy Matematicheskogo Instituta Imeni V. A. Steklova}, 326:173--192, 2024.

\bibitem{fan2024cohomology}
Feifei Fan.
\newblock The cohomology of the classifying space of $ {PU(4)} $.
\newblock {\em arXiv:2405.08256, to appear on Journal of Topology and Analysis}, 2024.

\bibitem{fan2024operators}
Feifei Fan.
\newblock Operators on symmetric polynomials and applications in computing the cohomology of {$BPU_n$}.
\newblock {\em arXiv:2410.11691}, 2024.

\bibitem{Fan25BPUp}
Feifei Fan.
\newblock A formula for the mod $p$ cohomology of {$BPU(p)$}.
\newblock arXiv:2503.23399, 2025.

\bibitem{garcia2019dai}
I\~naki Garc\'ia-Etxebarria and Miguel Montero.
\newblock Dai-{F}reed anomalies in particle physics.
\newblock {\em Journal of High Energy Physics}, (8):003, 77pp, 2019.

\bibitem{gu2019topological}
Xing Gu.
\newblock The topological period--index problem over 8-complexes, {I}.
\newblock {\em Journal of Topology}, 12(4):1368--1395, 2019.

\bibitem{gu2020topological}
Xing Gu.
\newblock The topological period-index problem over 8-complexes, {II}.
\newblock {\em Proceedings of the American Mathematical Society}, 148(10):4531--4545, 2020.

\bibitem{gu2019cohomology}
Xing Gu.
\newblock On the cohomology of the classifying spaces of projective unitary groups.
\newblock {\em Journal of Topology and Analysis}, 13(02):535--573, 2021.

\bibitem{gu_zzz}
Xing Gu, Yu~Zhang, Zhilei Zhang, and Linan Zhong.
\newblock The {$p$}-primary subgroups of the cohomology of {$BPU_n$} in dimensions less than {$2p+5$}.
\newblock {\em Proceedings of the American Mathematical Society}, 150(9):4099--4111, 2022.

\bibitem{Kameko2008brown}
Masaki Kameko and Nobuaki Yagita.
\newblock The {Brown-Peterson} cohomology of the classifying spaces of the projective unitary groups {$PU(p)$} and exceptional {Lie} groups.
\newblock {\em Transactions of the American Mathematical Society}, 360(5):2265--2284, 2008.

\bibitem{kono1975cohomology}
Akira Kono and Mamoru Mimura.
\newblock On the cohomology of the classifying spaces of {$PSU(4n+2)$} and {$PO(4n+2)$}.
\newblock {\em Publications of the Research Institute for Mathematical Sciences}, 10(3):691--720, 1975.

\bibitem{kono1975cohomology1}
Akira Kono, Mamoru Mimura, and Nobuo Shimada.
\newblock Cohomology of classifying spaces of certain associative {$H$}-spaces.
\newblock {\em Journal of Mathematics of Kyoto University}, 15(3):607--617, 1975.

\bibitem{Kono_Yagita}
Akira Kono and Nobuaki Yagita.
\newblock Brown-{P}eterson and ordinary cohomology theories of classifying spaces for compact {L}ie groups.
\newblock {\em Transactions of the American Mathematical Society}, 339(2):781--798, 1993.

\bibitem{toda1987cohomology}
Hiroshi Toda.
\newblock Cohomology of classifying spaces.
\newblock In {\em Homotopy theory and related topics (Kyoto, 1984)}, volume~9 of {\em Advanced Studies in Pure Mathematics}, pages 75--108. North-Holland, Amsterdam, 1987.

\bibitem{vavpetivc2005mod}
Ale{\v{s}} Vavpeti{\v{c}} and Antonio Viruel.
\newblock On the mod $p$ cohomology of {$BPU(p)$}.
\newblock {\em Transactions of the American Mathematical Society}, 357(11):4517--4532, 2005.

\bibitem{vezzosi2000chow}
Gabriele Vezzosi.
\newblock On the {C}how ring of the classifying stack of {${\rm PGL}_{3,\mathbb{C}}$}.
\newblock {\em Journal f\"ur die Reine und Angewandte Mathematik}, 523:1--54, 2000.

\bibitem{vistoli2007cohomology}
Angelo Vistoli.
\newblock On the cohomology and the {C}how ring of the classifying space of {$PGL_p$}.
\newblock {\em Journal f{\"u}r die reine und angewandte Mathematik}, 610:181--227, 2007.

\bibitem{zzz}
Yu~Zhang, Zhilei Zhang, and Linan Zhong.
\newblock The {$p$}-primary subgroup of the cohomology of {$BPU_n$} in dimension {$2p+6$}.
\newblock {\em Topology and its Applications}, 338:108642, 2023.

\bibitem{zhang2024}
Zhilei Zhang and Linan Zhong.
\newblock On the {$p$}-primary subgroups of the cohomology of the classifying spaces of {$PU_n$}.
\newblock {\em arXiv:2402.09103}, 2024.

\end{thebibliography}
\end{document}